\documentclass{article}  

\usepackage{setspace}
\usepackage{float}
\usepackage{amsmath,amsfonts,amssymb}
\usepackage{amsthm}
\usepackage{graphicx}
\usepackage{epstopdf}
\usepackage{caption}

\usepackage{a4wide}
\usepackage{color}
\usepackage{url}

\theoremstyle{plain}
\newtheorem{theorem}{Theorem}

\newtheorem{corollary}{Corollary}
\theoremstyle{definition}
\newtheorem{definition}{Definition}

\newtheorem{example}{Example}
\newtheorem{remark}{Remark}

\newcommand\blfootnote[1]{%
	\begingroup
	\renewcommand\thefootnote{}\footnote{#1}%
	\addtocounter{footnote}{-1}%
	\endgroup
}

\begin{document}	

\title{Gradient Mittag-Leffler and strong stabilizability\\ 
of time fractional diffusion processes\blfootnote{This is a preprint 
of a paper whose final and definite form is published in 
'Journal of Mathematical Sciences' ({\tt https://link.springer.com/journal/10958})}}

\author{Hanaa Zitane$^{1,2}$\\
\texttt{h.zitane@uae.ac.ma} 
\and Delfim F. M. Torres$^{2,}$\thanks{Corresponding author.}\\
\texttt{delfim@ua.pt}}

\date{$^{1}$Department of Mathematics, Faculty of Sciences,\\
University of Abdelmalek Essaadi, B.P.~2121, Tetouan, Morocco\\[0.3cm] 
$^{2}$\text{Center for Research and Development in Mathematics and Applications (CIDMA),} 
Department of Mathematics, University of Aveiro, 3810-193 Aveiro, Portugal}

\maketitle	


\begin{abstract}
This paper deals with the gradient stability and the gradient stabilizability 
of Caputo time fractional diffusion linear systems. First, we give sufficient 
conditions that allow the gradient Mittag-Leffler and strong stability, where 
we use a direct method based essentially on the spectral properties of the system dynamic. 
Moreover, we consider a class of linear and distributed feedback controls that 
Mittag-Leffler and strongly stabilize the state gradient. The proposed results 
lead to an algorithm that allows us to gradient stabilize the state of the 
fractional systems under consideration. Finally, we illustrate the effectiveness 
of the developed algorithm by a numerical example and simulations. 
	
\medskip
	
\noindent \textbf{Keywords:} decomposition approach, 
fractional diffusion systems, gradient stabilization, 
gradient stability, partial differential equations,
distributed controls.
	
\end{abstract}
\medskip

\noindent \textbf{2020 Mathematics Subject Classification}:  
26A33, 35R11, 93C05, 93D15, 93D20.	


\section{Introduction}

Fractional systems with Caputo derivative have gained significant 
attention in various fields, due to their physical meaning and 
wide applications, such as in epidemics \cite{Ye,ZZT}, porous 
and fractured media \cite{List}, population dynamics \cite{Berestycki}, etc.
The class of fractional diffusion systems (FDSs) has been especially widely
investigated in chemistry, physics (particle diffusion), biology, economics 
(diffusion of price values) and sociology (diffusion of people).  Indeed, 
based on the continuous-time random walks theory, it is worth noting that 
FDSs can be used to efficiently describe anomalous diffusion processes 
with a highly heterogeneous aquifer, and offer better performance 
than using conventional diffusion systems \cite{Asol1}. For example, 
FDSs may modelize the relaxation phenomena in complex viscoelastic materials 
\cite{Asol13}, the movement of plasma under high temperature and high pressure, 
and the power law decay of prime number distributions \cite{ChenLiang}.

Stability is among the most extensively studied concepts in control theory 
since an unstable dynamical system is of no interest as it can explode 
anywhere and at any instant. It signifies that the system remains in a 
constant state unless affected by an external action and returns 
to its equilibrium state when the external action is detached. 
Also one of the challenging problems related to the stability 
analysis is stabilizability, which means that if a dynamical system 
is not stable by itself, then the question arises whether 
it may be stabilized by selecting suitable controller inputs.

In recent years, the analysis of the stability, stabilizability, and 
related problems of FDSs have been tackled in many works: see, e.g., 
\cite{Gao,Chen2,Huang,ZitP21,ZitP22}. For instance, in \cite{ZitP21,ZitP22} 
the authors consider different degrees of the state stability 
and stabilization: the strong and the exponential ones of fractional 
distributed systems involving both Caputo and Riemann--Liouville time 
derivatives of order $q \in (0,1]$. Furthermore, the decomposition approach 
is applied to conceive the state stabilizing feedback law. Moreover, 
in \cite{Chen2}, the authors investigate the regional stabilization concept, 
where they present some results about the regional Mittag-Leffler 
and asymptotic stabilization of FDSs on a sub-region of the geometrical domain, 
while in \cite{Huang}, the boundary controller design and stability analysis 
for a time-space fractional diffusion equation are presented. Among other 
interesting concepts concerning the control theory of fractional diffusion systems 
one can find the gradient controllability \cite{Gea} and gradient observability \cite{Gea1} 
in a sub-region of the geometrical domain of the Caputo and Riemann--Liouville time FDSs. 

After analyzing the existing literature, it is evident that the gradient stability 
and stabilizability of FDSs are still untreated subjects in the literature 
and this fact is the motivation of the present work. Thus, the purpose 
of this paper is to derive some sufficient conditions that allow 
the gradient stability of Caputo fractional time systems of order 
$q \in (0, 1]$ and also to design a control law that ensures 
the gradient stabilization of a class of fractional diffusion systems, 
defined on a bounded and connected subset $\Omega\subset \mathbb{R}^{n}$, 
$n=1,2,3,\ldots$ (with smooth boundary $\partial \Omega$) and described by
\begin{equation}
\label{system0}
\left\{
\begin{array}{ll}
^{C}D_{t}^{q}y(x,t)=Ay(x,t)+Lv(x,t), & t\in]0,+\infty[,\\
y(\eta, t)=0,  & \eta \in\partial \Omega, t \in]0,+\infty[,\\
y(x,0)=y_{0}(x),  & y_{0}\in Y,
\end{array}
\right.
\end{equation}
where the state space is $Y=H^{1}(\Omega)$, 
the operator $A: D(A) \subset Y \longrightarrow Y$ is linear 
and generates a $C_{0}$-semi-group $(T(t))_{t \geq 0}$ on $Y$ \cite{Nagel}, 
$-A$ is a uniformly elliptic operator and $L$ is a linear bounded operator 
from $V_{ad}$ (the space of controls) 
into $Y$ with $v(t)$ a scalar input. Let $^{C}D_{t}^{q}$ be the left-sided 
Caputo fractional derivative \cite{der2} of order $q \in (0, 1]$   
with respect to time $t$, defined by
$$
^{C}D_{t}^{q}y(x,t)=\dfrac{1}{\Gamma(1-q)}
\int_{0}^{t}(t-s)^{-q} \dfrac{d}{ds}y(x,s)\, \mathrm{d}s.
$$
Moreover, we consider the gradient operator $\nabla$, which is defined by
$$
\begin{array}{llcll}
\nabla & : & Y & \mapsto & (L^{2}(\Omega))^{n} \\
& & y & \mapsto & \nabla y := \left(     
\dfrac{\partial y}{\partial x_{1}}, \dfrac{\partial y}{\partial x_{2}}, 
\ldots , \dfrac{\partial y}{\partial x_{n}}  \right).
\end{array}$$
In particular, when $q=1$, the gradient stabilizability of the system~\eqref{system0} 
reduces to the gradient stabilizability of a classical integer order diffusion system, 
which is investigated in \cite{Zerrik3,zit1}. It is worthy to mention that there 
are various applications of the gradient modeling. For instance, it is the exchange 
of the energy  problem between a casting plasma on a plane target, 
which is perpendicular to the direction of the flow sub-diffusion process 
from measurements carried out by internal thermocouples. Another example 
is the concentration regulation of a substrate at the upper bottom of a biological 
reactor sub-diffusion process, which is observed between two levels \cite{Gea1}. 
For a richer background on gradient models, we refer the reader to \cite{Cort,Kessell}. 
Indeed, gradient stabilization plays an important role in enhancing 
the reliability and efficiency of different engineering and physical systems. 
In particular, gradient stabilization  methods can be applied to various real-world problems. 
For instance, in aerospace engineering, when simulating airflow around aircraft, 
the presence of shock waves can introduce instabilities. Gradient stabilization helps 
to maintain numerical stability, enabling engineers to accurately predict lift and drag forces, 
which are vital for optimizing aircraft performance and fuel efficiency. Moreover, in structural 
engineering, maintaining the stability of structures under varying loads is crucial. 
Therefore, gradient stabilization methods can be applied to optimize the design 
of beams, bridges, and buildings \cite{Kirsch}.

The remainder contents of this paper are structured as follows. 
In Section~\ref{sec:2}, we examine the state gradient Mittag-Leffler 
and strong stability. In Section~\ref{sec:3}, under sufficient conditions, 
we present a class of distributed feedback controls that ensure the state 
gradient Mittag-Leffler and strong stabilization of fractional 
time diffusion systems using two approaches: the first one concerns 
the decomposition method and the second one is based on the completely 
monotonic property of Mittag-Leffler functions and also on the spectrum 
properties of the operator $A$. In Section~\ref{sec:4}, we illustrate 
our theoretical results by a numerical example and simulations. 
In the last section, Section~\ref{sec:5}, we give a conclusion 
and some possible directions of future research. 


\section{Gradient stability}
\label{sec:2}

Throughout this paper, the spaces $Y$ and $L^{2}(\Omega)^{n}$ 
are endowed with the usual inner products $\langle \cdot , \cdot \rangle$ 
and $<\cdot,\cdot>_{n}$ and the associated norms $\Arrowvert \cdot \Arrowvert$ 
and  $\Arrowvert \cdot \Arrowvert_{L^{2}(\Omega)^{n}}$, respectively. 
By $\nabla^{*}$ we denote the adjoint operator of the gradient operator $\nabla$. 
Moreover, to avoid confusion, we shall sometimes denote $y(\cdot,t):=y(t)$. 

In this section, we explore some sufficient conditions for the gradient 
Mittag-Leffler and strong stability of the fractional linear 
diffusion system~\eqref{system0} with the control operator $L=0$.

We first give the following gradient stability definitions. 

\begin{definition}
Let $y_{0} \in Y$. The system~\eqref{system0} is said to be 
\begin{itemize}
\item Gradient Mittag-Leffler stable, if there exist 
$C>0$, $\xi >0$, $b>0$ such that
$$
\Arrowvert \nabla y(t) \Arrowvert_{L^{2}(\Omega)}  
\leq C \{ E_{q}(- \xi t^{q})\}^{b}\Arrowvert y_{0} \Arrowvert.
$$
\item Gradient strongly stable, if the corresponding solution 
$y(\cdot)$ of \eqref{system0} satisfies
$$
\underset{t \longrightarrow +\infty }{\lim}
~\Arrowvert  \nabla y(t) \Arrowvert_{L^{2}(\Omega)}=0.
$$
\end{itemize} 
\end{definition}

Now, let us consider the sets
$$
\omega_{1}(A)=\left\{ \lambda~\in~\omega(A): \lambda \geq 0, 
N(A-\lambda I)\nsubseteq N(\nabla^{*} \nabla) \right\},
$$	
$$
\omega_{2}(A)=\left\{\lambda~\in~\omega(A): \lambda<0, 
N(A-\lambda I)\nsubseteq N(\nabla^{*} \nabla)\right\},
$$
where $\omega(A)$ is the spectrum of $A$ and $N(*)$ is the kernel of operator $*$.
\begin{theorem}
\label{theorem1}
Let $(\lambda_{n})_{n\geq1}$ and $(\phi_{n})_{n\geq1}$ be the eigenvalues 
and the corresponding eigenfunctions of the operator $A$ on $Y$. If 
\begin{description}
\item[i)] $\omega_{1}(A)=\emptyset$;

\item[ii)] for all $ \lambda_{n} \in \omega_{2}(A)$, $n=1,2,\ldots$, 
there exists $\xi>0$ such that $\lambda_{n}\leq -\xi$;
\end{description}
then, the system~\eqref{system0} is gradient Mittag-Leffler 
(respectively gradient strongly) stable in $\Omega$.
\end{theorem}

\begin{proof} 
We present the proof of the gradient Mittag-Leffler stability. 
The proof of the gradient strong stability 
is similar and is left to the reader.
The system~\eqref{system0} admits a unique weak solution 
$y \in C(0, T, Y)$ \cite{solmittag} defined by	
\begin{equation}
\label{key1}
y(t)= \sum_{n=1}^{+\infty} E_{q} (\lambda_{n} t^{q}) 
\langle y_{0}, \phi_{n}\rangle\phi_{n},~\forall y_{0}\in Y,
\end{equation}
where $E_{q} (\lambda_{n} t^{q})$ is the  Mittag-Leffler 
function in one parameter \cite{mettalefflerdef}, given by
$$
E_{q}(\lambda_{n} t^{q})=\sum_{k=0}^{+\infty}
\frac{(\lambda_{n} t^{q})^{k}}{\Gamma(q k+ 1)}.
$$	
One has
\begin{equation}
\label{eqgrad}
\nabla y(t)= \sum_{n=1}^{+\infty} 
E_{q}(\lambda_{n} t^{q}) \langle y_{0}, 
\phi_{n}\rangle \nabla\phi_{n},
\end{equation}	 
which yields	
$$
\Arrowvert \nabla y(t) \Arrowvert_{L^{2}(\Omega)}^{2} 
= \displaystyle\sum \limits_{{n=1}}^{+\infty}(E_{q}(\lambda_{n} 
t^{q}))^{2} \langle y_{0},\phi_{n}\rangle^{2}
\Arrowvert \nabla \phi_{n} \Arrowvert_{L^{2}(\Omega)}^{2}.
$$
By virtue of assumption~(ii) and by using the fact 
that $\dfrac{d}{dt}  E_{q}(-t^{q})\leq 0$, 
for $q \in (0,1)$ and $ t>0$ \cite{MAINARDI2}, it follows that 
$$ (-\lambda_{n})^{\frac{1}{q}}> \xi^{\frac{1}{q}}>0$$
and
$$E_{q}(\lambda_{n}t^{q})=E_{q}(-((-\lambda_{n})^{\frac{1}{q}}t)^{q})
\leq E_{q}(-\xi t^{q}), \lambda_{n} \in \omega_{2}(A).
$$
Furthermore, from the previous relation 
and conditions (i) and (ii), one gets
$$
\Arrowvert \nabla y(t) \Arrowvert_{L^{2}(\Omega)} ^{2} 
\leq (E_{q}(- \xi t^{q}))^{2} \displaystyle\sum \limits_{{n=1}}^{+\infty} 
\langle y_{0},\phi_{n}\rangle^{2}\Arrowvert 
\nabla \phi_{n} \Arrowvert_{L^{2}(\Omega)} ^{2},
$$
which implies that there exists $C>0$ such that
$$
\Arrowvert \nabla y(t) \Arrowvert_{L^{2}(\Omega)}  
\leq C E_{q}(-\xi t^{q})\Arrowvert y_{0} \Arrowvert.
$$
The proof is complete.
\end{proof}

To highlight the previous result, we consider the following example.

\begin{example}
Let us consider $\Omega=]0, 1[\times ]0, 1[ \subset \mathbb{R}^{2}$ 
and the following sub-diffusion system:
\begin{equation}
\label{system55cp}
\left\{
\begin{array}{lll}
^{C}D_{t}^{0.6}y(x_{1}, x_{2}, t)=(\dfrac{\partial^2  y}{\partial x_{1}^2}
+\dfrac{\partial^2 y}{\partial x_{2}^2})(x_{1}, x_{2}, t),
&  (x_{1}, x_{2}) \in\Omega, t \in]0,+\infty[,\\
y(\eta_{1}, \eta_{2}, t)=0,  & (\eta_{1}, \eta_{2}) \in\partial \Omega, t \in]0,+\infty[,\\
y(x_{1}, x_{2},0)=z_{0}(x_{1}, x_{2}), &(x_{1}, x_{2}) \in \Omega,
\end{array}
\right.
\end{equation}
where the dynamic $A=(\dfrac{\partial^2}{\partial x_{1}^2}
+\dfrac{\partial^2}{\partial x_{2}^2})$, 
the eigenvalues are
\begin{equation}
\label{eqsol}
\lambda_{nm}= -(n+m)^{2}\pi^{2}, ~n,m\geq1
\end{equation}
with the corresponding  eigenfunctions 
$$
\phi_{nm}(x_{1}, x_{2})=\sqrt{2}\sin(n\pi x_{1})\sin(m\pi x_{2}), ~n,m\geq1.
$$
The state of system~\eqref{system55cp} is defined by
$$
z(x_{1}, x_{2},t)=\sum_{n=1}^{+\infty} 
\sum_{m=1}^{+\infty}E_{0.6} (\lambda_{nm} t^{0.6})
\langle z_{0}, \phi_{nm}\rangle \phi_{nm}(x_{1}, x_{2}),
$$
and the state gradient of system~\eqref{system55cp} is given by 
\begin{equation}
\label{ff}
\nabla z(x_{1}, x_{2},t)=\sum_{n=1}^{+\infty} 
\sum_{m=1}^{+\infty}E_{0.6} (\lambda_{nm} t^{0.6})
\langle z_{0}, \phi_{nm}\rangle\nabla \phi_{nm}(x_{1}, x_{2}).
\end{equation}
By virtue of \eqref{eqsol}, it follows that
$$
\lambda_{nm}\leq -4\pi^{2},~ \text{ for all } n,~m \geq 1,
$$
which means that $\omega_{1}(A)=\varnothing$.
Hence, from Theorem~\ref{theorem1}, we conclude with
the Mittag-Leffler stability of the state gradient~\eqref{ff} 
of system~\eqref{system55cp}.	
\end{example}


\section{Gradient stabilizability characterizations}
\label{sec:3}

Consider the system~\eqref{system0} with the same conditions.

\begin{definition}
The system~\eqref{system0} is said to be gradient Mittag-Leffler 
(respectively gradient strongly) stabilizable if, for any $y_0 \in Y$,  
there exists a bounded operator $D \in \mathcal{L}(Y,V_{ad})$ 
such that the closed-loop system	
\begin{equation}
\label{system2}
\left\{
\begin{array}{ll}
^{C}D_{t}^{q}y(x,t)=(A+LD)y(x,t), & t\in]0,+\infty[,\\
y(\eta, t)=0,  & \eta \in\partial \Omega, t \in]0,+\infty[,\\
y(x,0)=y_{0}(x),  & y_{0}\in Y,
\end{array}
\right.
\end{equation} 	
is gradient Mittag-Leffler (respectively gradient strongly) stable.
\end{definition}

\begin{remark}
\begin{enumerate}
\item The state gradient stabilization is a special case 
of the output stabilization with the output operator being the gradient $\nabla$.

\item  The gradient Mittag-Leffler stabilization 
implies the gradient strong stabilization.

\item The gradient state stabilization is cheaper than the  
state stabilization. Indeed, if we consider the functional cost
$$
J( v)= \displaystyle \int_{0}^{+\infty}
\Arrowvert v(t) \Arrowvert^{2}\, \mathrm{d}t
$$
and the feedback spaces
$$
V^{1}_{ad}= \left\{ v \in L^{2}(0, +\infty; V_{ad})| v~ 
\text{ stabilizes~ the~ gradient~ state~ of}~\eqref{system0} \right\}, 
$$	 
and
$$
V^{2}_{ad}= \left\{ v \in L^{2}(0, +\infty; V_{ad})| 
v~ \text{ stabilizes~ the~  state~ of}~\eqref{system0} \right\},
$$
then one has that 
$$
V^{2}_{ad}\subset V^{1}_{ad}.
$$
Hence,
$$
\underset{u \in V^{1}_{ad} }{\min}J(v)
\leq \underset{v \in V^{2}_{ad} }{\min}J(v).
$$
\end{enumerate}
\end{remark}

Next, we shall characterize the gradient stabilizing control 
of system~\eqref{system0} using the decomposition method, 
which can be described as follows:

Consider a fixed $\beta> 0$ and suppose that the operator $A$ 
is self-adjoint with compact resolvent, which means that there 
are at most finitely many nonnegative eigenvalues of $A$ 
and each with finite dimensional eigenspace, i.e., 
there exists $l \in \mathbb{N}$ such that
\begin{equation}
\label{spectrumdecom}
\omega(A)= \omega_{u}(A)\cup \omega_{s}(A),
\end{equation}
where 
$$
\omega_{u}(A)
= \left\{ \lambda_{n} \in \omega(A), ~~~n=1, 2,\ldots,l \right\}
$$
and 
$$
\omega_{s}(A)=  \{\lambda_{n} \in \omega(A), ~~~ n=l+1, l+2 \ldots \}
$$
with $\lambda_{l} \geq 0$ and $\lambda_{l+1} \leq -\beta$. 
Since the sequence $(\phi_{n})_{n\geq1}$ forms a complete 
and orthonormal basis in $Y$, it follows that the state 
space $Y$ can be decomposed  according to
\begin{equation}
\label{eqa}
Y=Y_{u}\oplus Y_{s},
\end{equation}
where 
$$
Y_{u}=PY=span\{ \phi_{1}, \phi_{2}, \ldots, \phi_{l} \},~ 
Y_{s}=(I-P)Y=span\{ \phi_{l+1}, \phi_{l+2}, \ldots \},
$$ 
with $P \in L(Y)$ being the projection operator, defined by 
$$
P=\frac{1}{2\pi}\int_{\varGamma} (\lambda I-A)^{-1}
\, \mathrm{d}\lambda,
$$
with $\varGamma$ being a curve surrounding $\omega(A)$ \cite{decomp}.

As a consequence, system~\eqref{system0} may be decomposed 
into the following two sub-systems:
\begin{equation}
\label{system3}
\begin{cases}^{c}D_{t}^{q}y_{u}(x,t)=A_{u}y_{u}(x,t)+PLv(x,t), \\
y_{u}(\eta,t)=0\\
y_{0u}=Py_0,\\
y_{u}=Py,
\end{cases}
\end{equation}
and 
\begin{equation}
\label{system4}
\begin{cases}^{c}D_{t}^{q}y_{s}(t)=A_{s}y_{s}(x,t)+(I-P)Lv(x,t),\\
y_{s}(\eta,t)=0\\
y_{0s}=(I-P)y_0,\\
y_{s}=Py,
\end{cases}
\end{equation} 
where $A_{s}$ and $A_{u}$ are the restrictions of $A$ on $Y_{s}$ 
and $Y_{u}$, respectively, with $A_{u}$ bounded on $Y_{u}$, 
$\omega(A_{u})=\omega_{u}(A)$ and $\omega(A_{s})=\omega_{s}(A)$.

The following theorem shows that the gradient stabilization of system
\eqref{system0} is equivalent to the gradient stabilization 
of system~\eqref{system3}.

\begin{theorem}
\label{theorem3}	
Assume that the spectrum $\omega (A)$ of $A$ admits the decomposition assumption 
\eqref{spectrumdecom} for some $\beta> 0$. If there exists a bounded operator 
$D_{u}\in L(Y_{u},V_{ad})$ such that the control 	
\begin{equation}
\label{contDecom}
v(t)=D_{u} y_{u}(t)
\end{equation} 
gradient Mittag--Leffler stabilizes system~\eqref{system3},	
then the system~\eqref{system0} is gradient Mittag-Leffler stabilizable 
by the feedback control \eqref{contDecom}.
\end{theorem}

\begin{proof} 
Since system~\eqref{system3} is gradient Mittag-Leffler stabilizable on $\Omega$, 
it follows that there exists $C>0$ such that
$$
\Arrowvert \nabla y_{u}(t) \Arrowvert_{L^{2}(\Omega)}  
\leq C  E_{q}(- \beta t^{q})\Arrowvert y_{0u} \Arrowvert,
$$
for some $\beta>0$. Moreover, from the relation (see \cite{MAINARDI2})
$$
\frac{1}{1+\Gamma(1-q)t^{q}}	
\leq E_{q}(- t^{q}) \leq \frac{1}{1+\Gamma(1+q)^{-1}t^{q}},
$$
it follows that
\begin{equation}
\label{mittagrel}
\Arrowvert \nabla y_{u}(t) \Arrowvert_{L^{2}(\Omega)}  
\leq C_{\beta} t^{-q}\Arrowvert y_{0u} \Arrowvert,
\end{equation}	
and
\begin{equation}
\label{mittagrel24}
\Arrowvert y_{u}(t) \Arrowvert_{L^{2}(\Omega)}  
\leq h_{\beta} t^{-q}\Arrowvert y_{0u} \Arrowvert,
\end{equation}
for some $h_{\beta}>0$ and $C_{\beta}= C \Gamma(1+q)\beta^{-1}$.
Furthermore, by virtue of inequality \eqref{mittagrel24}, 
one obtains that \eqref{contDecom} implies
\begin{equation}
\label{conts}
\Arrowvert v(t)\Arrowvert \leq h_{\beta} t^{-q} 
\Arrowvert D_{u}\Arrowvert \Arrowvert y_{0u} \Arrowvert.
\end{equation}	
On the other hand, the unique mild solution of 
system~\eqref{system0} \cite{solmittag} can be written as 
\begin{equation}
\label{ccc}
y_{s}(t)= \sum_{n=l+1}^{+\infty} E_{q} (\lambda_{n} t^{q}) 
\langle y_{0s}, \phi_{n}\rangle\phi_{n}+ \sum_{n=l+1}^{+\infty} 
\displaystyle \int_{0}^{t}(t-s)^{q-1}E_{q,q}(\lambda_{n}(t-s)^{q})
\langle (I-P)Lv(s), \phi_{n}\rangle\phi_{n} \, \mathrm{d}s,
\end{equation} 
with $E_{q,\alpha} (\lambda_{n} t^{q})$ being the two parameter
Mittag-Leffler function \cite{mettalefflerdef}, given by
$$
E_{q,\alpha}(z)=\sum_{k=1}^{+\infty}\frac{z^{k}}{\Gamma(q k+ \alpha)},
\quad Re(q)>0,~\alpha>0, 
\quad z \in \mathbb{C}.
$$
Therefore, formula \eqref{ccc} leads to 
$$ 
\nabla y_{s}(t)= \sum_{n=l+1}^{+\infty} E_{q} (\lambda_{n} t^{q}) 
\langle y_{0s}, \phi_{n}\rangle \nabla\phi_{n}
+ \sum_{n=l+1}^{+\infty} \displaystyle \int_{0}^{t}(t-s)^{q-1}
E_{q,q}(\lambda_{n}(t-s)^{q})\langle (I-P)Lv(s), 
\phi_{n}\rangle \nabla \phi_{n} \, \mathrm{d}s. 
$$	
Now, feeding system~ \eqref{system4} by the same control 
$v(t)=D_{u} y_{u}(t)$, it follows, by using  \eqref{conts}, 
assumption \eqref{spectrumdecom}, and the monotonic property 
of the Mittag-Leffler function \cite{MAINARDI2}, that
$$
\begin{tabular}{rcll}
$\Arrowvert \nabla y_{s}(t) \Arrowvert$ 
& $\leq$ & $GE_{q}(- \beta t^{q}) \Arrowvert y_{0s} \Arrowvert 
+ G h_{\beta} \Arrowvert D_{u} \Arrowvert  
\Arrowvert I-P \Arrowvert  \Arrowvert L \Arrowvert 
\displaystyle \int_{0}^{t}(t-s)^{q-1}s^{-q}
E_{q,q}(-\beta(t-s)^{q} \, \mathrm{d}s$\\
&$\leq $& $GE_{q}(- \beta t^{q}) \Arrowvert y_{0s} 
\Arrowvert + G h_{\beta} \Arrowvert D_{u} \Arrowvert  
\Arrowvert I-P \Arrowvert  \Arrowvert L \Arrowvert 
\Arrowvert y_{0s} \Arrowvert
\displaystyle\sum \limits_{{k=1}}^{+\infty}\displaystyle\int_{0}^{t} 
\dfrac{(-\beta)^{k}(t-s)^{q k+q-1}s^{-q} \, \mathrm{d}s}{\Gamma(q k+ q)}$\\
&$\leq $& $GE_{q}(- \beta t^{q}) \Arrowvert y_{0s} \Arrowvert 
+ G h_{\beta} \Arrowvert D_{u} \Arrowvert  \Arrowvert I-P \Arrowvert  
\Arrowvert L \Arrowvert\Arrowvert y_{0s} \Arrowvert  
\displaystyle\sum \limits_{{k=1}}^{+\infty} 
\dfrac{(-\beta)^{k}t^{q k}\Gamma(1-\mu)}{\Gamma(q k-1)}$\\
&$\leq $& $  GE_{q}(- \beta t^{q}) \Arrowvert y_{0s} \Arrowvert 
+ G h_{\beta} \Gamma(1-\mu)\Arrowvert D_{u} \Arrowvert  
\Arrowvert I-P \Arrowvert  \Arrowvert L \Arrowvert  
\Arrowvert y_{0s} \Arrowvert
E_{q}(- \beta t^{q})).$
\end{tabular}
$$
One obtains	that
\begin{equation}
\label{mittagrel2}
\Arrowvert \nabla y_{s}(t) \Arrowvert
\leq  G (1+h_{\beta} \Arrowvert D_{u} \Arrowvert  
\Arrowvert I-P \Arrowvert  \Arrowvert L \Arrowvert)	
E_{q}(- \beta t^{q})\Arrowvert y_{0s} \Arrowvert.
\end{equation}
Hence, by replacing \eqref{mittagrel} and \eqref{mittagrel2} 
in the inequality
\begin{equation}
\label{infinity3}
\Arrowvert \nabla y(t) \Arrowvert 
\leq \Arrowvert \nabla y_{s}(t) \Arrowvert 
+ \Arrowvert  \nabla y_{u}(t) \Arrowvert, 
\end{equation}
one gets the Mittag-Leffler stabilization of 
system~\eqref{system0} by the control law
$(D_{u} y_{u}(t),0)$.
\end{proof}

\begin{corollary}
Under the same assumptions of Theorem~\ref{theorem3}, if
system~\eqref{system3} is strongly stabilizable using 
the control~\eqref{contDecom}, then system~\eqref{system0}
is gradient strongly stabilizable by the feedback 
control~\eqref{contDecom} on $\Omega$.	 
\end{corollary}

The following result is a generalization of Theorem~\ref{theorem1}, 
where the spectral properties of the operator $A+LD$ 
are used to examine the gradient stabilization.

\begin{theorem}
\label{theorem2}
Let $(\gamma_{n})_{n\geq1}$ and $(\varphi_{n})_{n\geq1}$ 
be the eigenvalues and the corresponding eigenfunctions 
of the operator $A+LD$ on $Y$. If 
\begin{description}
\item[i)] $\omega_{1}(A+LD)=\emptyset$;
\item[ii)] for all $ \gamma_{n} \in \omega_{2}(A+LD)$, 
$n=1,2,\ldots$, there exists $\xi>0$ satisfying $\gamma_{n}\leq -\xi$;
\end{description}
then, system~\eqref{system0} is gradient Mittag-Leffler 
(gradient strongly, respectively) stabilizable 
by the following feedback control:
\begin{equation}
\label{contsim}
v(t)=Dy(t).
\end{equation}
\end{theorem}

\begin{proof} 
It is an immediate consequence of the proof of Theorem~\ref{theorem1}.
\end{proof}

The above results lead to the following algorithm for the stabilization 
of the state gradient of system~\eqref{system0} 
by the feedback control~\eqref{contsim}. 


\paragraph{Algorithm}
\begin{description}
\item[$\odot$] Initialization: 
\begin{description}
\item[$\triangleright$] Threshold accuracy $\epsilon > 0$.
\item[$\triangleright$] Initial condition $y_{0}$.
\item[$\triangleright$] Fractional order $q$.
\end{description}	
\item[$\odot$]  Repeat:
\begin{description}
\item[$\triangleright$] Apply the feedback control $u(\cdot)$, 
given by \eqref{contsim}, to system~\eqref{system0};
\item[$\triangleright$] Solve system~\eqref{system0} using formula 
$$
y(\cdot)= \sum_{n=1}^{+\infty} 
E_{q} (\gamma_{n} t^{q}) 
\langle y_{0}, \varphi_{n}\rangle\varphi_{n}.
$$
\item[$\triangleright$] Compute the state gradient $\nabla y(\cdot)$ 
of system~\eqref{system0};
\end{description}
until $\Arrowvert \nabla y(\cdot) \Arrowvert < \epsilon$.	
\end{description}


\section{Application and numerical simulations}
\label{sec:4}

In this section, we shall present some
numerical illustrations of the developed algorithm.

Let $\Omega=]0, 1[$ and consider the fractional 
diffusion controlled system of order $q\in (0,1)$ given by 
\begin{equation}
\label{systemSim}
\left\{
\begin{array}{lll}
^{C}D_{t}^{q}y(x, t)=\dfrac{\partial^2 y}{\partial x^2}(x, t)
+\pi^2 y(x, t)+Lv(t),& x \in \Omega, t \in ]0,+\infty[,\\
y(0, t)=y(1, t)=0,  & t \in]0,+\infty[,\\
y(x,0)=y_0, &x \in \Omega,
\end{array}
\right.
\end{equation}
where operator $A=\dfrac{\partial^2}{\partial x^2}+\pi^{2}$ with the domain
$$
D(A)=\left\{ y\in H^{1}(\Omega), y(0, t)=y(1, t)=0, (\forall~t>0)\right\}.
$$
The eigenvalues of $A$ are given by $\lambda_{n}=-n^{2}\pi^{2}+\pi^2$, $n\geq1$, 
associated to the eigenfunctions $\phi_{n}(x)= \sqrt{2}\sin((n\pi x))$, $n\geq1$.

We now take the control operator $L=\pi I$. Then the feedback control~\eqref{contsim} 
with  $D=-\pi I_{Y}$ ($I_{Y}$ is the identity operator of $Y=H^{1}(\Omega)$), 
which yields
\begin{equation}
\label{Contexpj}
v(t)=-\pi^2 y(t).
\end{equation}
Furthermore, the state and the gradient state  
of system~\eqref{systemSim} are given respectively by 
$$
y(x,t)=\sum_{n=1}^{+\infty} E_{q} (\gamma_{n} t^{q})
\langle y_{0}, \varphi_{n}\rangle\varphi_{n}(x),
$$
and it follows that
\begin{equation}
\label{grad}
\nabla y(x,t)=\sum_{n=1}^{+\infty} 
E_{q} (\gamma_{n} t^{q})\langle y_{0}, 
\varphi_{n}\rangle \nabla\varphi_{n}(x),
\end{equation}
with 
$$
\gamma_{n}\in\omega(A+LD)=\left\{-n^{2}\pi^{2}, ~n\geq1 \right\},
$$
which implies that 
$$
\omega_{1}(A+LD)=\emptyset
$$ 
and there exists $\xi=\pi^{2}>0$ such that, 
for all $ \gamma_{n} \in \omega_{2}(A+LD)$, we have
$$
\gamma_{n}\leq-\pi^{2},~n\geq1.
$$
Therefore, from Theorem~\ref{theorem2}, the state gradient~\eqref{grad} 
of system~\eqref{systemSim} is Mittag-Leffler stabilizable on $\Omega=]0, 1[$.

We first give the simulations for the case $q=0.9$ 
and take the initial state $y_0=x^{2}(x-1)$.

\begin{figure}[H]
\begin{center}
\includegraphics[scale=0.8]{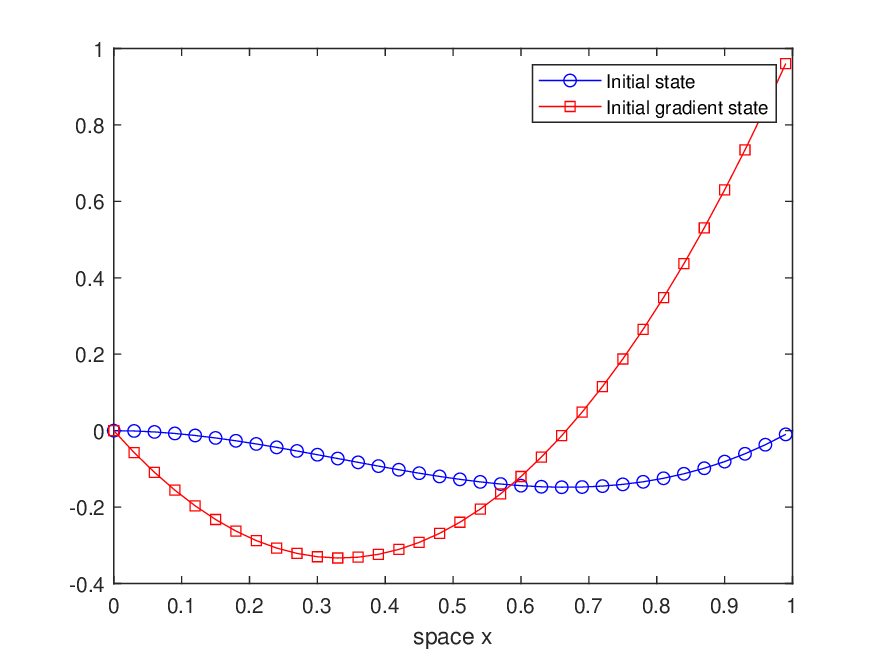}
\caption{The initial state and initial gradient state of 
system~\eqref{systemSim} behavior, for $q=0.9$.}\label{Figure1}
\end{center}
\end{figure}
\begin{figure}[H]
\begin{center}
\includegraphics[scale=0.8]{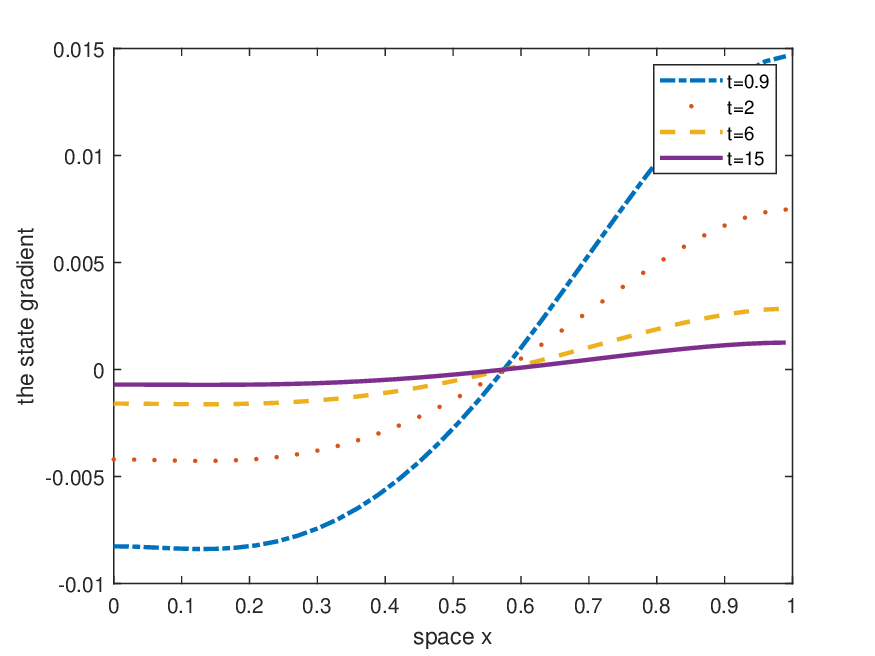}
\caption{The gradient state of system~\eqref{systemSim} evolution, 
with $q=0.9$, for instants $t=0.9$, $t=2$, $t=6$, 
and $t=15$.}\label{Figure2}
\end{center}
\end{figure}

From Figure~\ref{Figure1}, we note that the state gradient 
is unstable at the initial instant $t=0$. However, Figure~\ref{Figure2} 
shows that the state gradient of system~\eqref{systemSim} evolves 
close to zero when time $t$ increases. Moreover, it is stabilized by 
the feedback law~\eqref{Contexpj}, from the instant $t=15$, 
with the gradient stabilization error equal to $4.1 \times 10^{-3}$.
\begin{figure}[H]
\begin{center}
\includegraphics[scale=0.8]{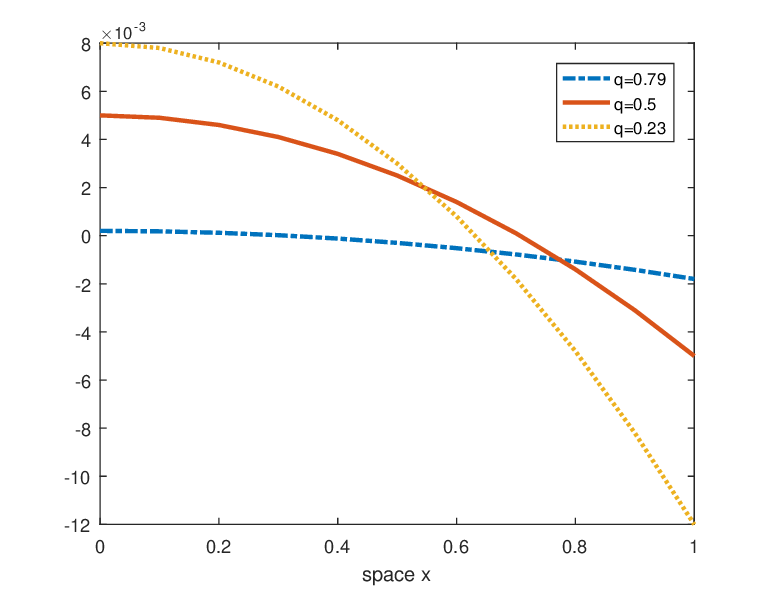}
\caption{The gradient state of system~\eqref{systemSim} evolution, 
at the instant $t=20$, for $q=0.79 $, $q=0.5 $ and $q=0.23$.}\label{Figure3}
\end{center}
\end{figure}
\begin{table}
\center
\begin{tabular}{l c c}
\hline
Fractional order $q$ & Gradient stabilization error\\
\hline	
$q= 0.1$ & $1.9142$\\
		
$q=0.23$ & $8.718\times 10^{-1}$\\
		
$q=0.5$ & $1.428\times 10^{-1}$\\
		
$q=0.65$ & $4.49 \times 10^{-2}$\\
		
$q=0.79$  &  $1.25 \times 10^{-2}$   \\
		
$q=0.9$ & $ 3.2 \times 10^{-3}$   \\
\hline
\end{tabular}
\caption{The gradient stabilization error of system~\eqref{systemSim} 
for some value of the fractional order $q$ at $t=20$.}\label{Table1}
\end{table}
Table~\ref{Table1} presents the variation effect of the fractional order $q$  
of system~\eqref{systemSim} on the state gradient stabilization error 
at a fixed instant $t=20$. We see that the gradient stabilization error 
decreases when the value of the fractional order $q$ increases. 
This is also illustrated by Figure~\ref{Figure3} 
in the case of the fractional orders $q=0.79 $, $q=0.5 $ and $q=0.23$.


\section{Conclusion}
\label{sec:5}

In this paper, the state gradient stability and stabilization 
of Caputo fractional diffusion systems are introduced. 
Sufficient conditions to obtain the gradient Mittag-Leffler 
and strong stability are investigated. Also, according 
to the conditions verified by the state space
and those verified by the dynamic of the fractional linear system, 
the decomposition method is applied to characterize the state 
gradient stabilizing feedback law. Furthermore, in order to validate 
our developed results, a numerical example with simulations is provided.

Several questions are still open and deserve further investigations. 
This is the case, for example, of studying the gradient stabilization 
of fractional diffusion systems involving other kinds 
of fractional derivatives with singular and nonsingular kernels, including 
the Riemann--Liouville fractional derivative, the Caputo--Fabrizio fractional 
derivative, the weighted Atangana--Baleanu fractional derivative, the power 
and tempered fractional derivatives, etc. It would be
also interesting to extend our Mittag-Leffler and strong gradient stabilization 
results to nonlinear fractional distributed systems. This is under investigation
and will be addressed elsewhere.


\section*{Declarations}

\subsection*{Conflict of interest}

The authors declare no conflict of interest.


\subsection*{Data availability}

The manuscript has no associated data.


\subsection*{Acknowledgements}

Zitane and Torres are supported by The Center for Research and
Development in Mathematics and Applications (CIDMA)
through the Portuguese Foundation for Science and Technology 
(FCT -- Funda\c{c}\~{a}o para a Ci\^{e}ncia e a Tecnologia),
projects UIDB/04106/2020 (\url{https://doi.org/10.54499/UIDB/04106/2020})
and UIDP/04106/2020 (\url{https://doi.org/10.54499/UIDP/04106/2020}).



\end{document}